\documentclass[12pt,a4paper]{amsart}
\usepackage{mathrsfs}

\newtheorem{theo+}              {Theorem}           [section]
\newtheorem{prop+}  [theo+]     {Proposition}
\newtheorem{coro+}  [theo+]     {Corollary}
\newtheorem{lemm+}  [theo+]     {Lemma}
\newtheorem{exam+}  [theo+]     {Example}
\newtheorem{rema+}  [theo+]     {Remark}
\newtheorem{defi+}  [theo+]     {Definition}

\newenvironment{theorem}{\begin{theo+}}{\end{theo+}}

\newenvironment{corollary}{\begin{coro+}}{\end{coro+}}

\newenvironment{definition}{\begin{defi+}}{\end{defi+}}

\usepackage{amsthm}
\theoremstyle{plain} \theoremstyle{remark}
\newtheorem{remark}{Remark}
\newtheorem{example}{Example}

\def \r{\mbox{${\mathbb R}$}}
\def\E{/\kern-1.0em \equiv }

\def\sdo{pseudo-Riemannian\;}
\title{ Biharmonic submanifolds of pseudo-Riemannian manifolds}
\author{Yuxin Dong$^*$ and Ye-Lin Ou$^{**}$}
\address{Institute of Mathematics, \newline\indent Fudan University, \newline\indent
Shanghai 200433,
\newline\indent 
P. R. China
\newline\indent 
Email: yxdong@fudan.edu.cn
\newline\indent 
\newline\indent 
Department of
Mathematics,\newline\indent Texas A $\&$ M University-Commerce,
\newline\indent Commerce TX 75429,\newline\indent USA.\newline\indent
E-mail:yelin.ou@tamuc.edu}
\thanks{$^*$ Supported by NSFC grant No. 11271071, and LMNS, Fudan}
\thanks{$^{**}$ Partially supported by the Texas A $\&$ M University-Commerce
Faculty Development Program (2015). The author is also grateful to Shanghai Center for Mathematical Sciences for the support of a two-month visit during which the work was done.}
\begin{document}
\title[Biharmonic submanifolds of pseudo-Riemannian manifolds ]{Biharmonic submanifolds of pseudo-Riemannian manifolds}
\date{12/02/2015}
\maketitle
\section*{Abstract}
\begin{quote} In this paper, we derived biharmonic equations for pseudo-Riemannian submanifolds of pseudo-Riemannian manifolds which includes the biharmonic equations for submanifolds of Riemannian manifolds as a special case. As applications, we proved that a pseudo-umbilical biharmonic pseudo-Riemannian submanifold of a pseudo-Riemannian manifold has constant mean curvature, we completed the classifications of biharmonic pseudo-Riemannian hypersurfaces with at most two distinct principal curvatures, which were used to give four construction methods to produce proper biharmonic pseudo-Riemannian submanifolds from minimal submanifolds. We also made some comparison study between biharmonic hypersurfaces of Riemannian space forms and the space-like biharmonic hypersurfaces of pseudo-Riemannian space forms.
{\footnotesize }
\end{quote}
\section{Biharmonic maps between \sdo manifolds}
All manifolds, maps, tensor fields are assumed to be smooth in this paper.\\
\indent {\bf 1.1 Some basic concepts and notations from \sdo \\geometry}\\
\indent Let $(M^m_t, g)$ be a pseudo-Riemannian (or semi-Riemannian) manifold of dimension $m$ with a nondegenerate metric with index $t$. Here, nondegeneracy means the only vector $X\in T_x M$ satisfying $g_x(X, Y)=0$ for all $Y\in T_xM$ and all $x\in M$ is $X=0$. We use $|X|=|g(X,X)|^{1/2}$ to denote the norm of a vector $X$ and call a vector a unit vector if its norm is $1$. A set of $m$ orthogonal unit vectors is called an orthonormal basis of $(M^m_t, g)$ at a point. A local orthonormal frame of $(M^m_t, g)$ is a set of $m$ local vector fields $\{e_i\}$ such that $g(e_i, e_j)=\epsilon_i\delta_{ij}$ with $\epsilon_1=\cdots =\epsilon_t=-1, \;\epsilon_{t+1}=\cdots =\epsilon_m=1$.\\
\indent For a local orthonormal frame $\{e_i\}$ on a neighborhood $U$ of $(M^m_t, g)$, we have
\begin{eqnarray}
X= \sum_{i=1}^m\epsilon_ig(X, e_i)e_i, \;\; {\rm for\; any}\; X\in \mathfrak{X}(U),\\
{\rm grad} f=\sum_{i=1}^m\epsilon_i (df(e_i))e_i=\sum_{i=1}^m\epsilon_i (e_i f)e_i
\end{eqnarray}
\begin{eqnarray}
{\rm div} X=\sum_{i=1}^m\epsilon_i g(\nabla_{e_i}X, e_i)\\
\Delta f=\sum_{i=1}^m\epsilon_i [e_i(e_i f)-(\nabla_{e_i}e_i) f]\\
{\rm The\; trace\; of\; a\; bilinear\; form\; b\; is\; given\; by}: {\rm Trace}_g b=\sum_{i=1}^m\epsilon_i b(e_i, e_i).
\end{eqnarray}
\indent In this paper, we will use the following notations and conventions.
\begin{itemize}
\item An $n$-dimensional pseudo-Euclidean space with index $s$ is denoted by \\$\r^n_s=(\r^n, \langle, \rangle)$ with $\langle x, y\rangle=-\sum_{i=1}^s x_iy_i+\sum_{k>s}^n x_ky_k$.
\item An $n$-dimensional pseudo-Riemannian sphere, denoted by $S^n_s(r) $, is defined to be $S^n_s(r)=\{ x\in \r^{n+1}_s: \langle x, x\rangle=r^2\}$, which, with the induced metric from $\r^{n+1}_s$, is a complete pseudo-Riemannian manifold with index $s$ and of constant sectional curvature $\frac{1}{r^2}$.
\item An $n$-dimensional pseudo-hyperbolic space, denoted by $H^n_s(r) $, is defined to be $H^n_s(r)=\{ x\in \r^{n+1}_{s+1}: \langle x, x\rangle=-r^2\}$, which, with the induced metric from $\r^{n+1}_{s+1}$, is a complete pseudo-Riemannian manifold with index $s$ and of constant sectional curvature $-\frac{1}{r^2}$.
\item A pseudo-Riemannian space form refers to one on the three spaces:\\ $\r^n_s$, $S^n_s(r)$, $H^n_s(r) $.
\item Sometimes, for the convenience of a unified treatment, we also use $N^n_s(C)$ to denote a pseudo-Riemannian space form of sectional curvature $C$, so $N^n_s(C)$ can be represented as one of the following model spaces:
\begin{equation}\label{SF}
N^n_s(C)=
\begin{cases} \r^n_s, \hskip1.5cm\;{\rm for}\; C=0,\\
S^n_s(1/\sqrt{C}),\;\;\;\; {\rm for}\; C>0,\; {\rm or}\\H^n_s(1/\sqrt{-C}),\; \;\;\;{\rm for}\; C<0.
\end{cases}
\end{equation}
\end{itemize}
Pseudo-Riemannian space forms have important applications in the theory of general relativity as it is well known that the pseudo-Riemannian space forms $\r^n_1$, $S^n_1(r)$, $H^n_1(r) $ are model spaces for Minkowski, de Sitter, and Anti-de Sitter space-times respectively.\\

{\bf 1.2 Harmonic and biharmonic maps between Riemannian manifolds}\\
\indent A map between Riemannian manifolds $\varphi:(M^m, g)\longrightarrow (N^n, h)$
is harmonic if its tension field 
\begin{equation}
\tau(\varphi)={\rm Trace}_{g}\nabla {\rm d} \varphi=\sum_{i=1}^m [\nabla^{\phi}_{e_1}\,d\phi (e_i)-d\phi(\nabla^M_{e_i}e_i)]
\end{equation}
vanishes identically.\\
\indent A {\em biharmonic map} between Riemannian manifolds is a map $\varphi:(M, g)\longrightarrow (N, h)$
which is a critical point of the
bienergy functional
\begin{equation}\nonumber
E_{2}\left(\varphi,\Omega \right)= \frac{1}{2} {\int}_{\Omega}
\left|\tau(\varphi) \right|^{2}{\rm d}x
\end{equation}
for every compact subset $\Omega$ of $M$. Using the first variational formula (
\cite{Ji2}) one sees that $\varphi$ is biharmonic if and only if
its bitension field vanishes identically, i.e.,
\begin{equation}\label{BI1}
\tau_{2}(\varphi):=-\triangle^{\varphi}(\tau(\varphi)) - {\rm
Trace}_{g} R^{N}({\rm d}\varphi, \tau(\varphi)){\rm d}\varphi
=0,
\end{equation}
where
\begin{equation}\notag
\triangle^{\varphi}=-{\rm Trace}_{g}(\nabla^{\varphi})^{2}= -{\rm
Trace}_{g}(\nabla^{\varphi}\nabla^{\varphi}-\nabla^{\varphi}_{\nabla^{M}})
\end{equation}
is the Laplacian on sections of the pull-back bundle $\varphi^{-1}
TN$ and $R^{N}$ is the Riemann curvature operator of $(N, h)$ defined by
$$R^{N}(X,Y)Z=
[\nabla^{N}_{X},\nabla^{N}_{Y}]Z-\nabla^{N}_{[X,Y]}Z.$$ 
\indent A submanifold of a Riemannian manifold is called a {\em biharmonic submanifold} if the isometric immersion that defines the submanifold is a biharmonic map. \\
\indent By definition, a harmonic map is always a biharmonic map. This, together with the well-known fact that a submanifold is minimal if and only if the isometric immersion that defines the submanifold is harmonic, implies that a minimal submanifold is always a biharmonic submanifold. So, it is a custom to call a biharmonic map, which is not harmonic, a {\em proper biharmonic map}, and {\em proper biharmonic submanifolds} are reserved for those biharmonic submanifolds, which are not minimal.\\

{\bf 1.3 Harmonic and biharmonic maps between pseudo-Riemannian manifolds}\\
\indent The generalization of the concepts of harmonic and biharmonic maps between Riemannian manifolds to the case of \sdo manifolds is straightforward. \\
\indent A map $\phi: (M^m_t, g)\longrightarrow (N^n_s, h)$ between pseudo-Riemannian manifolds is a harmonic map if its tension field vanishes identically, i.e.,
\begin{eqnarray}
\tau(\phi)={\rm Trace}_g\nabla d\phi=\sum_{i=1}^m \epsilon_i[\nabla^{\phi}_{e_1}\,d\phi (e_i)-d\phi(\nabla^M_{e_i}e_i)]\equiv 0,
\end{eqnarray}
where $\{e_i|\; i=1,2,\cdots m\}$ is a local orthonormal frame of $M^m_t$ with\\ $\langle e_i, e_k\rangle=\epsilon_i\delta_{ik},\; \epsilon_1=\cdots =\epsilon_t=-1, \;\epsilon_{t+1}=\cdots =\epsilon_m=1$.\\
\indent Similarly, we have
\begin{definition}
A map $\phi: (M^m_t, g)\longrightarrow (N^n_s, h)$ between pseudo-Riemannian manifolds is a {\em biharmonic map} if its bi-tension field vanishes identically, i.e.,
\begin{eqnarray}
\tau_2(\phi)=\sum_{i=1}^m\epsilon_i\left((\nabla^{\phi}_{e_i}\nabla^{\phi}_{e_i}-\nabla^{\phi}_{\nabla^{M}_{e_i}e_i})\tau(\phi)
- R^{N}({\rm d}\phi ({e_i}), \tau(\phi)){\rm d}\phi(e_i)\right)\equiv 0,
\end{eqnarray}
where $\{e_i|\; i=1,2,\cdots m\}$ is a local orthonormal frame of $M^m_t$ with\\ $\langle e_i, e_k\rangle=\epsilon_i\delta_{ik},\; \epsilon_1=\cdots =\epsilon_t=-1, \;\epsilon_{t+1}=\cdots =\epsilon_m=1$.\\
\end{definition}
\begin{remark}
Note that the only difference of the tension (and the bitension) fields between the Riemannian and the \sdo cases lies in the definition of the trace of a bilinear form in these two different cases.
\end{remark}
\begin{definition}
A \sdo submanifold $M^m_t$ of a \sdo \\manifold $(N^n_s, h)$ is said to be biharmonic if the isometric immersion $M^m_t\longrightarrow (N^n_s,h)$ is a biharmonic map. Proper biharmonic maps and proper biharmonic submanifolds in the \sdo case are similarly defined as in the Riemannian case.
\end{definition}
{\bf 1.4 Some recent work on biharmonic pseudo-Riemannian submanifolds}\\
\indent The study of biharmonic submanifolds was initiated in the middle of 1980s by B. Y. Chen \cite{Ch1} in his program to understand finite types submanifolds, and independently by G. Y. Jiang \cite{Ji1}, \cite{Ji2}, \cite{Ji3} in his effort to study the geometry of $k$-polyharmonic maps proposed by Eells and Lemaire in \cite{EL}. Since 2001,  biharmonic submanifolds of Riemannian space forms have been vigorously studied by many geometers around the world with the main focus to classify and/or characterize biharmonic submanifolds in a Riemannian space form. We refer the readers to the  recent survey \cite{Ou2} of the second named author of the paper for some history, fundamental problems, current results and open problems with updated references about the submanifolds of the Riemannian space forms.\\
\indent Comparatively, not much study on biharmonic submanifolds in pseudo-Riemannian manifolds has been made except a few works on some classifications of biharmonic submanifolds which can be summarized as follows. \\
\indent Chen-Ishikawa \cite{CI} studied biharmonic pseudo-Riemannian submanifolds in pseudo-Euclidean space $\r^4_s$ for $s=1, 2, 3$. They classified all space-like biharmonic curves and proved that any \sdo biharmonic surface in Minkowski space $\r^3_1$ is minimal, i.e., $H=0$. They also obtained some classification results on pseudo-Riemannian biharmonic surfaces $x : M^2_1\longrightarrow \r^4_s$ with $s=1, 2,3$ under the assumption of nonzero constant mean curvature or mean curvature vector being light-like. Ouyang \cite{Oy} studied space-like biharmonic submanifolds of pseudo-Riemannian space forms. Among other things, he derived the equations, in local coordinates, for space-like biharmonic submanifolds in a \sdo manifold, and applied it to study space-like biharmonic submanifolds of \sdo space forms with parallel mean curvature vector fields. He also proved that a space-like biharmonic hypersurface with constant mean curvature in $\r^n_1$, $S^n_1(C)$ has to be minimal, i.e., $H=0$, and that a space-like biharmonic hypersurface with constant mean curvature and $|A|^2\ne \frac{n}{C^2}$ in $H^n_1(C)$ has to be minimal, i.e., $H=0$. Defever-Kaimakamis-Papantoniou \cite{DKP} proved that any biharmonic \sdo hypersurface in $4$-dimensional pseudo-Euclidean space with diagonalizable second fundamental form has to be minimal. Later in 2008, Chen \cite{Ch2} obtained a complete classification of biharmonic Lorentzian flat surfaces in pseudo-Euclidean space $\r^4_2$. Zhang in \cite{Zh} shown that the only space-like proper biharmonic surface in a 3-dimensional pseudo-Riemannian space form $N^3_1(C)$ is (a part) of $H^2(\frac{1}{\sqrt{2}})$; He also proved that both $H^n(\frac{1}{\sqrt{2}})$ and $H^p(\frac{1}{\sqrt{2}})\times H^q(\frac{1}{\sqrt{2}})$ with $p\ne q, \; p+q=n$ are space-like proper biharmonic hypersurfaces in the Anti-de Sitter space $H^{n+1}_1(1)$. Sasahara in \cite{Sa} studied pseudo-Riemannian submanifolds in $3$-dimensional Lorentzian space forms of nonzero constant curvature $ N^3_1(C),\;\; C\ne 0$. He gave complete classifications of proper biharmonic curves $\gamma: I\longrightarrow S^3_1(1)$ and $\gamma: I\longrightarrow H^3_1(1)$. He also proved that a pseudo-Riemannian Surface $x: M^2_t\longrightarrow N^3_1$ is proper biharmonic if and only if it is a part of the following: $S^2_1(\frac{1}{\sqrt{2}})\subseteq S^3_1(1)$, a B-scroll over a curve and of $K_M=2$ in $S^3_1(1)$, or $H^2 (\frac{1}{\sqrt{2}})\subseteq H^3_1(1)$. Fu \cite{Fu} classified proper biharmonic pseudo-Riemannian submanifold $x : M^m\longrightarrow \r^n_s\; (n\ge 4, s\ge 1)$ with parallel mean curvature vector field $H$ and diagonalizable $A_{H}$ to be marginally trapped. Most recently, Liu-Du \cite{LD} classified proper biharmonic pseudo-Riemannian hypersurfaces with at most two distinct principal curvatures in a \sdo Riemannian space form. Finally, Liu-Du-Zhang in \cite{LDZ} gave a classification of space-like biharmonic submanifolds of a \sdo space form, which are pseudo-umbilical or have parallel mean curvature vector fields.\\

 In this paper, we derived biharmonic equations for pseudo-Riemannian submanifolds of pseudo-Riemannian manifolds (Theorem \ref{MT1}) which includes the biharmonic equations for submanifolds of Riemannian manifolds obtained in \cite{BMO2} and biharmonic equations for \sdo submanifolds of \sdo space forms stated in \cite{LD} as  special cases. As applications, we proved that a pseudo-umbilical biharmonic pseudo-Riemannian submanifold of a pseudo-Riemannian manifold has constant mean curvature (Theorem \ref{MT2}) generalizing the corresponding results of the Riemannian cases obtained in \cite {Di}, \cite{CMO2} and \cite{BMO2}, we gave a complete classification of biharmonic pseudo-Riemannian hypersurfaces of \sdo space forms with at most two distinct principal curvatures (Theorem \ref{MT3} and Corollary \ref{SSn}) , which improve the classifications given in \cite{LD} and include the classification obtained by Balmus-Montaldo-Oniciuc in \cite{BMO1} as special cases. The complete classifications (Corollary \ref{SSn}) were used to give four construction methods to produce proper biharmonic pseudo-Riemannian submanifolds from minimal submanifolds (Theorems \ref{GZ}, \ref{GZ1} and \ref{GZ2}). We also made some comparison study between biharmonic hypersurfaces of Riemannian space forms and the space-like biharmonic hypersurfaces of pseudo-Riemannian space forms.

\section{Biharmonic pseudo-Riemannian submanifolds of pseudo-Riemannian manifolds}
\begin{theorem}\label{MT1}
A pseudo-Riemannian submanifold $\phi: (M^m_t, g)\hookrightarrow (N^n_s, h)$ of a pseudo-Riemannian manifold is biharmonic if and only if its mean curvature vector field $H$, the second fundamental form $B$ and the Weingarten operator $A_H$ solve the following system of PDEs
\begin{equation}\label{BSM}
\begin{cases}
\Delta^{\perp} H+{\rm Trace}\, B(A_H(\cdot), \cdot)+[{\rm Trace}\,{\rm R}^N(d \phi (\cdot), H)d\phi (\cdot)]^\perp =0,\\
{\rm Trace} \,A_{\nabla^{\perp}_{(\cdot)}H}\,(\cdot) +\frac{m}{4} \,{\rm grad}\, \langle H, H\rangle
+[{\rm Trace}\,{\rm R}^N(d \phi (\cdot), H)d\phi (\cdot)]^\top=0,
\end{cases}
\end{equation}
where 
\begin{eqnarray}
\Delta^{\perp} H&=& -\sum_{i=1}^m \epsilon_i (\nabla^{\perp}_{e_1} \nabla^{\perp}_{e_1}-\nabla^{\perp}_{\nabla^M_{e_i}e_1})H,\\
{\rm Trace}\,B(A_H(\cdot), \cdot) &=& \sum_{i=1}^m\,\epsilon_i B(A_H(e_i), e_i),\\
{\rm Trace} \,A_{\nabla^{\perp}_{(\cdot)}H}\,(\cdot) &=& \sum_{i=1}^m\,\epsilon_i A_{\nabla^{\perp}_{e_i}H}\,(e_i)
\end{eqnarray}
for a local orthonormal frame $\{e_i|\; i=1,2,\cdots m\}$ of $M^m_t$ with $\langle e_i, e_k\rangle=\epsilon_i\delta_{ik}$ and $\epsilon_1=\cdots =\epsilon_t=-1, \;\epsilon_{t+1}=\cdots =\epsilon_m=1$.
\end{theorem}
\begin{proof}
It is well known that the tension field of the isometric immersion $\phi: (M^m_t, g) \hookrightarrow (N^n_s, h)$ is given by $\tau(\phi)=mH$, where $H$ is the mean curvature vector of the pseudo-Riemannian submanifold. It follows that the \sdo submanifold is biharmonic if and only if
\begin{equation}\label{B1}
\sum_{i=1}^m\epsilon_i\,\{(\nabla^N_{e_i}\nabla^N_{e_i}-\nabla^N_{\nabla^M_{e_i}e_i})H-R^N(d\phi(e_i), H)d\phi(e_i)\}=0.
\end{equation}
\indent A straightforward computation using Formulas of Gauss and Weingarten for submanifolds yields
\begin{eqnarray}\notag
\sum_{i=1}^m \,\nabla^N_{e_i}\nabla^N_{e_i}\,H &=& \sum_{i=1}^m \,\nabla^N_{e_i}[-A_H(e_i)+\nabla^{\perp}_{e_i}H]\\\notag
&=& \sum_{i=1}^m \,\nabla^N_{e_i}[-A_H(e_i)+\nabla^{\perp}_{e_i}H]\\\label{B2}
&=& \sum_{i=1}^m [-\nabla^M_{e_i}A_H(e_i)-B(A_H(e_i), e_i)-A_{\nabla^{\perp}_{e_i}H}e_i+\nabla^{\perp}_{e_i}\nabla^{\perp}_{e_i}H],\\\label{B3}
\sum_{i=1}^m \,\nabla^N_{\nabla^M_{e_i}e_i}\,H &=& -\sum_{i=1}^m \, A_H(\nabla^M_{e_i}e_i)+\sum_{i=1}^m \,\nabla^{\perp}_{\nabla^M_{e_i}e_i}H.
\end{eqnarray}
\indent Recall that 
\begin{eqnarray}
(\nabla_XA_H)(Y) &=& \nabla_X(A_H(Y))-A_H(\nabla_XY),\\
(\nabla^N_XB)(Y,Z)&=&\nabla^{\perp}_XB(Y,Z)-B(\nabla_X Y, Z)-B(Y, \nabla_X Z),
\end{eqnarray}
 and we have the Codazzi-Mainardi equation
\begin{equation}
(\nabla^N_XB)(Y,Z)=(\nabla^N_YB)(X,Z)+({\rm R}^N(X,Y)Z)^{\perp}
\end{equation}
for a submanifold.\\
\indent Now we compute
\begin{eqnarray}\notag
&&\langle (\nabla^N_XB)(Y,Z), H\rangle \\\notag
&=&\langle \nabla^{\perp}_XB(Y,Z), H\rangle-\langle B(\nabla_X Y, Z), H\rangle-\langle B(Y, \nabla_X Z), H\rangle\\\notag
&=& X \langle B(Y,Z), H\rangle-\langle B(Y,Z), \nabla^{\perp}_X H\rangle-\langle B(\nabla_X Y, Z), H\rangle-\langle B(Y, \nabla_X Z), H\rangle\\\notag
&=& X \langle A_H(Y), Z\rangle-\langle A_{\nabla^{\perp}_XH}Y, Z\rangle-\langle A_H(\nabla_X Y), Z\rangle-\langle A_H(Y), \nabla_X Z \rangle\\\notag
&=& \langle \nabla_X(A_H(Y)), Z\rangle-\langle A_{\nabla^{\perp}_XH}Y, Z\rangle-\langle A_H(\nabla_X Y), Z\rangle\\\notag
&=& \langle (\nabla_X A_H) (Y), Z\rangle-\langle A_{\nabla^{\perp}_XH}Y, Z\rangle.
\end{eqnarray}
Using this, we have
\begin{eqnarray}\notag
\langle (\nabla^M_{e_i}A_H)(e_i), e_j\rangle &=&
\langle (\nabla^N_{e_i}B)(e_i, e_j), H\rangle+\langle A_{\nabla^{\perp}_{e_i}H}e_i, e_j\rangle\\\notag
&=&\langle (\nabla^N_{e_j}B)(e_i, e_i) +({\rm R}^N(e_i,e_j)e_i)^{\perp}, H\rangle+\langle A_{\nabla^{\perp}_{e_i}H}e_i, e_j\rangle\\\notag
&=&\langle \nabla^{\perp}_{e_j}B (e_i, e_i)-2B(e_i, \nabla^M_{e_j}e_i), H\rangle+\langle {\rm R}^N(d\phi(e_i),H)d\phi(e_i), e_j\rangle\\
&&+\langle A_{\nabla^{\perp}_{e_i}H}e_i, e_j\rangle.\\\notag
\end{eqnarray}
By choosing a local orthonormal frame with the property that at a point $x\in M_t^m$ we have $\nabla^M_{e_j}e_i=0$ for all $i, j=1, 2, \cdots, m$, then we have
\begin{eqnarray}\notag
\sum_{i=1}^m\langle (\nabla^M_{e_i}A_H)(e_i), e_j\rangle_x
&=&m \langle \nabla^{\perp}_{e_j} H, H\rangle_x+ \sum_{i=1}^m \{\langle {\rm R}^N(d\phi(e_i),H)d\phi(e_i), e_j\rangle+\langle A_{\nabla^{\perp}_{e_i}H}e_i, e_j\rangle\}_x\\
&=&\left(\frac{m}{2} e_j \langle H, H\rangle+ \sum_{i=1}^m \{\langle {\rm R}^N(d\phi(e_i),H)d\phi(e_i), e_j\rangle+\langle A_{\nabla^{\perp}_{e_i}H}e_i, e_j\rangle\}\right)_x.
\end{eqnarray}
As this is true for any point $x\in M^m_t$ we obtain
\begin{eqnarray}\notag
\sum_{i=1}^m\langle (\nabla^M_{e_i}A_H)(e_i), e_j\rangle=\frac{m}{2} e_j \langle H, H\rangle + \sum_{i=1}^m \langle {\rm R}^N(d\phi(e_i),H)d\phi(e_i), e_j\rangle+\sum_{i=1}^m\langle A_{\nabla^{\perp}_{e_i}H}e_i, e_j\rangle.
\end{eqnarray}
It follows that
\begin{eqnarray}\notag
\sum_{i=1}^m\epsilon_i(\nabla^M_{e_i}A_H)(e_i) &=& \sum_{j=1}^m\epsilon_j \langle \sum_{i=1}^m\langle \epsilon_i (\nabla^M_{e_i}A_H)(e_i), e_j\rangle\,e_j \\\notag
&=&\frac{m}{2} \sum_{j=1}^m\epsilon_j e_j (\langle H, H\rangle)e_j+ \sum_{i,j=1}^m \epsilon_j\epsilon_i\langle A_{\nabla^{\perp}_{e_i}H}e_i, e_j\rangle e_j\\\notag &&+\sum_{i,j=1}^m \epsilon_j \epsilon_i\langle {\rm R}^N(d\phi(e_i),H)d\phi(e_i), e_j\rangle e_j\\\label{B4}
&=&\frac{m}{2} {\rm grad}\langle H, H\rangle+ \sum_{i=1}^m\epsilon_i A_{\nabla^{\perp}_{e_i}H}e_i +[{\rm Trace}\,{\rm R}^N(d \phi (\cdot), H)d\phi (\cdot)]^\top.
\end{eqnarray}
\indent Note that the Riemann curvature decomposes into its normal and the tangential parts, 
\begin{eqnarray}\notag
&&\sum_{i=1}^m \;\epsilon_i R^N(d\phi(e_i), H)d\phi(e_i) =\\\label{B5} && [{\rm Trace}\,{\rm R}^N(d \phi (\cdot), H)d\phi]^\perp 
+[{\rm Trace}\,{\rm R}^N(d \phi (\cdot), H)d\phi]^\top.
\end{eqnarray}
\indent Substituting (\ref{B2}), (\ref{B3}), (\ref{B4}), and (\ref{B5}) into the second equation of (\ref{B1}) and comparing the normal and the tangential components we obtain the theorem.
\end{proof}
\begin{remark}
One can check that the biharmonic equation for \sdo \\submanifolds of \sdo space forms given in Theorem \ref{MT1} generalizes the biharmonic equation for Riemannian submanifolds of Riemannian manifolds obtained by Chen \cite{Ch1}, \cite{Ch2}, \cite{Ch3}, \cite{CI}, Caddeo-Montaldo-Oniciuc \cite{CMO2} and Balmus-Montaldo-Oniciuc \cite{BMO2}, it also includes the biharmonic equation for \sdo \\submanifolds of pseudo-Euclidean spaces studied by Chen in \cite{Ch1}, \cite{Ch2}, \cite{Ch3}, \cite{CI}, the biharmonic equation for space-like submanifolds of \sdo space form obtained by Ouyang in \cite{Oy}, and the biharmonic equations for \sdo submanifolds of a \sdo space form stated in \cite{LD} as special cases.
\end{remark}
As an immediate consequence, we have the following corollary, which was stated in \cite{LD} without a proof.
\begin{corollary}
A pseudo-Riemannian submanifold $\phi: (M^m_t, g)\hookrightarrow (N^n_s(C), h)$ of a pseudo-Riemannian space form of sectional curvature $C$ is biharmonic if and only if its mean curvature vector field $H$, the second fundamental form $B$ and the Weingarten operator $A_H$ solve the following system of PDEs
\begin{equation}\label{BSMC}
\begin{cases}
\Delta^{\perp} H+{\rm Trace}\, B(A_H(\cdot), \cdot)-mC\,H =0,\\
{\rm Trace} \,A_{\nabla^{\perp}_{(\cdot)}H}\,(\cdot) +\frac{m}{4} \,{\rm grad}\, \langle H, H\rangle
=0.
\end{cases}
\end{equation}
\end{corollary}
\begin{proof}
Note that when the ambient space is a \sdo space form of constant sectional curvature $C$, we have
\begin{eqnarray}
\sum_{i=1}^m \;\epsilon_i\,R^N(d\phi(e_i), H)d\phi(e_i) &=&-\sum_{i=1}^m\,\epsilon_i\, C\langle \phi(e_i), \phi(e_i)\rangle\,H\\
&=&-C\sum_{i=1}^m \;\epsilon_i^2=-mCH.
\end{eqnarray}
This, together with Theorem \ref{MT1}, gives the corollary.
\end{proof}

For the case of codimension one, we have the following corollary, which gives the biharmonic equation of a generic \sdo hypersurfaces of a generic \sdo manifold generalizing the biharmonic hypersurface equation in a generic Riemannian manifold obtained in \cite{Ou1}.
\begin{corollary}\label{CD1}
Let $\phi:(M^{m}_t, g)\longrightarrow (N^{m+1}_s, h)$ be a \sdo hypersurface in a \sdo manifold with mean curvature vector field $H=f \xi$, where $\xi$ is a unit normal vector field with $\langle \xi, \xi\rangle=\epsilon_{m+1}=\pm 1$ . Then
$\phi$ is biharmonic if and only if
\begin{equation}
\begin{cases}
\Delta f-\epsilon_{m+1}f (|A_\xi|^{2}-{\rm
Ric}^N(\xi,\xi))=0,\\
A_\xi\,({\rm grad}\,f )+\epsilon_{m+1}\frac{m}{2} f {\rm grad}\, f
-\, f \,({\rm Ric}^N\,(\xi))^{\top}=0,
\end{cases}
\end{equation}
where ${\rm Ric}^N : T_xM\longrightarrow T_xM$ denotes the Ricci operator of
the ambient space defined by $\langle {\rm Ric}\, (Z), W\rangle={\rm
Ric}^N (Z, W)$.
\end{corollary}
\begin{proof}
Note that for a \sdo hypersurface $\phi:(M^{m}_t, g)\longrightarrow (N^{m+1}_s, h)$ with mean curvature vector $H=f\xi$, we have
\begin{eqnarray}\label{H1}
\Delta^{\perp} H&=&-(\Delta f)\xi,\\\label{H2}
{\rm Trace}\,B(A_H(\cdot), \cdot) &=& \sum_{i=1}^m\,\epsilon_i B(A_H(e_i), e_i)=\epsilon_{m+1} f |A_\xi|^2\xi,\\\label{H3}
({\rm Trace}\,{\rm R}^N(d \phi (\cdot), H)d\phi)^\perp &=&-[\epsilon_{m+1} f\,{\rm Ric}^N(\xi,\xi)]\xi\\\label{H4}
({\rm Trace}\,{\rm R}^N(d \phi (\cdot), H)d\phi)^\top &=& -f ({\rm Ric}^N\,(\xi))^{\top}\\\label{H5}
{\rm Trace} \,A_{\nabla^{\perp}_{(\cdot)}H}\,(\cdot) &=& \sum_{i=1}^m\,\epsilon_i A_{\nabla^{\perp}_{e_i}H}\,(e_i)=A_\xi({\rm grad} f).
\end{eqnarray}
Substituting (\ref{H1}), (\ref{H2}), (\ref{H3}), (\ref{H4}), and (\ref{H5}) into biharmonic \sdo submanifold equation (\ref{BSM}) we obtain the corollary.
\end{proof}
When the ambient space is a \sdo space form, Corollary \ref{CD1} reduces to the following corollary, which was obtained in \cite{LD}.
\begin{corollary}\label{CD1C}
A \sdo hypersurface $\phi:(M^{m}_t, g)\longrightarrow (N^{m+1}_s(C), h)$ of a \sdo space form with mean curvature vector $H=f \xi$, where $\xi$ is a unit normal vector field with $\langle \xi, \xi\rangle=\epsilon_{m+1}=\pm 1$, is biharmonic if and only if
\begin{equation}
\begin{cases}
\Delta f-f (\epsilon_{m+1}\,|A_\xi|^{2}-cm)=0,\\
A_\xi\,({\rm grad}\,f )+\epsilon_{m+1} \frac{m}{2} f {\rm grad}\, f
=0.
\end{cases}
\end{equation}
\end{corollary}

\section{Some classifications and constructions of biharmonic \sdo submanifolds}

In this section, we will first prove that a pseudo-umbilical biharmonic \sdo submanifold of a  \sdo manifold has constant mean curvature. We then give a complete classification of biharmonic \sdo  \\hypersurfaces of \sdo space forms with at most two distinct principal curvatures, and finally we use the classifications to give four methods to construct proper biharmonic \sdo submanifolds using precompositions of minimal submanifolds.\\
\indent Recall that \sdo submanifold $\phi: M^m_t\hookrightarrow N^n_s$ is said to be pseudo-umbilical if its shape operator $A_H$ with respect to the mean curvature vector field $H$ is given by 
\begin{equation}\label{SU}
A_H=\langle H, H\rangle I.
\end{equation}

Similar to the definition of a biconservative hypersurface defined in \cite{MOR}, we define a {\em biconservative submanifold} to be a submanifold whose tangential component of the bitension field vanishes identically. It is clear from the definition that any biharmonic submanifold is biconservative. \\
\indent Now we are ready to prove the following theorem which generalizes the corresponding results in the Riemannian case given in \cite{BMO2}.

\begin{theorem}\label{MT2}
An $m$-dimensional ($m\ne 4$) pseudo-umbilical \sdo \\submanifold $\phi: M^m_t\hookrightarrow N^n_s$  is biconservative if and only if it has constant mean curvature. In particular, If an $m$-dimensional ($m\ne 4$) pseudo-umbilical \sdo submanifold $\phi: M^m_t\hookrightarrow N^n_s$  is biharmonic, then it has constant mean curvature. 
\end{theorem}
\begin{proof}
Using (\ref{B4}) we have
\begin{eqnarray}\notag
\sum_{i=1}^m\epsilon_i A_{\nabla^{\perp}_{e_i}H}e_i +[{\rm Trace}\,{\rm R}^N(d \phi (\cdot), H)d\phi (\cdot)]^\top&=& \sum_{i=1}^m\epsilon_i(\nabla^M_{e_i}A_H)(e_i)-\frac{m}{2} {\rm grad}\langle H, H\rangle\\\notag
&=& \sum_{i=1}^m\epsilon_i(\nabla^M_{e_i}( A_H (e_i))-\frac{m}{2} {\rm grad}\langle H, H\rangle\\\label{GD11}
&=& (1-\frac{m}{2}) {\rm grad}\langle H, H\rangle,
\end{eqnarray}
where in obtaining the second equality we have used, as in \cite{CMO2}, the normal coordinates at a point and assuming $\{e_i\}$ are the corresponding vector fields, whilst in obtaining the third equality we have used pseudo-umbilical condition (\ref{SU}).\\
\indent Substituting (\ref{GD11}) into the second equation of the biharmonic submanifold equation (\ref{BSM}) we have
\begin{equation}
(4-m){\rm grad}\langle H, H\rangle=0,
\end{equation}
from which we conclude that ${\rm grad}\langle H, H\rangle=0$ for $m\ne 4$. Thus, we obtain the theorem.
\end{proof}

Now we  give a complete  classification of pseudo-Riemannian hypersurfaces with at most two distinct principal curvatures in a \sdo space form, which improves the classifications given in \cite{LD}.
\begin{theorem}\label{MT3}
A proper biharmonic \sdo hypersurface $M^n_t\longrightarrow N^{n+1}_s(C)$ with diagonalizable shape operator that has at most two distinct principal curvatures is a part of the following
\begin{itemize}
\item[(i)] $S^n_s(\frac{1}{\sqrt{2C}})\subseteq S^{n+1}_s(\frac{1}{\sqrt{C}})$ or $S^p_t(\frac{1}{\sqrt{2C}})\times S^{n-p}_{s-t}(\frac{1}{\sqrt{2C}})\subseteq S^{n+1}_s(\frac{1}{\sqrt{C}})$ with $C>0,\;n\ne 2p$ and 
\item[(ii)] $H^n_{s-1}(\frac{1}{\sqrt{-2C}})\subseteq H^{n+1}_s(\frac{1}{\sqrt{-C}})$ or $H^p_{t-1}(\frac{1}{\sqrt{-2C}})\times H^{n-p}_{s-t}(\frac{1}{\sqrt{-2C}})\subseteq H^{n+1}_s(\frac{1}{\sqrt{-C}})$ with $C<0, n\ne 2p$.
\end{itemize}
\end{theorem}
\begin{proof}
By the classification of Liu-Du \cite{LD}, we know that a nondegenerate proper biharmonic hypersurface $M^n_t\longrightarrow N^{n+1}_s(C)$ with diagonalizable shape operator that has at most two distinct principal curvatures is a part of one of the following hypersurfaces:
\begin{itemize}
\item[(i)] $S^n_s(\frac{1}{\sqrt{2C}})\subseteq S^{n+1}_s(\frac{1}{\sqrt{C}})$ or $S^p_t(\frac{1}{\sqrt{C_1}})\times S^{n-p}_{s-t}(\frac{1}{\sqrt{C_2}})\subseteq S^{n+1}_s(\frac{1}{\sqrt{C}})$ \\with $C>0, n\ne 2p$ and $C_1, C_2$ satisfying $\frac{1}{C_1}+\frac{1}{C_2}=\frac{1}{C}, \; p^2C_1+(n-p)^2C_2\ne n^2C$, and $pC_1+(n-p)C_2=2nC$,
\item[(i)] $H^n_{s-1}(\frac{1}{\sqrt{-2C}})\subseteq H^{n+1}_s(\frac{1}{\sqrt{-C}})$ or $H^p_{t-1}(\frac{1}{\sqrt{-C_1}})\times H^{n-p}_{s-t}(\frac{1}{\sqrt{-C_2}})\subseteq H^{n+1}_s(\frac{1}{\sqrt{-C}})$ \\with $C<0, n\ne 2p$ and $C_1, C_2<0$ satisfying $\frac{1}{C_1}+\frac{1}{C_2}=\frac{1}{C}, \; p^2C_1+(n-p)^2C_2\ne n^2C$, and $pC_1+(n-p)C_2=2nC$.
\end{itemize}
Solving the equations
\begin{eqnarray}
\begin{cases}
\frac{1}{C_1}+\frac{1}{C_2}=\frac{1}{C}, \\
p^2C_1+(n-p)^2C_2\ne n^2C,\\
pC_1+(n-p)C_2=2nC
\end{cases}
\end{eqnarray}
for $C_1$ and $C_2$ we have the unique solution $C_1=C_2=2C$. Thus, we obtain the theorem.
\end{proof}
As an immediate consequence, we have
\begin{corollary}\label{SSn}
A proper biharmonic \sdo hypersurface $M^n_t\longrightarrow N^{n+1}_s(1)$ with diagonalizable shape operator that has at most two distinct principal curvatures is a part of one of the following hypersurfaces:
\begin{itemize}
\item[(i)] $S^n_s(\frac{1}{\sqrt{2}})\subseteq S^{n+1}_s(1)$ or $S^p_t(\frac{1}{\sqrt{2}})\times S^{n-p}_{s-t}(\frac{1}{\sqrt{2}})\subseteq S^{n+1}_s(1)$ with $n\ne 2p$ and 
\item[(ii)] $H^n_{s-1}(\frac{1}{\sqrt{2}})\subseteq H^{n+1}_s(1)$ or $H^p_{t-1}(\frac{1}{\sqrt{2}})\times H^{n-p}_{s-t}(\frac{1}{\sqrt{2}})\subseteq H^{n+1}_s(1)$ with $ n\ne 2p$.
\end{itemize}
\end{corollary}
\begin{example}
Proper biharmonic hypersurfaces of $ N_s^{4}(C)$ having diagonalizable shape operator with at most two principal curvatures can be listed as follows:
\begin{itemize}
\item[I.] {\em Riemannian Cases}:
\begin{itemize}
\item[(i)] $\mathbb{R}^4, \;\;\;H^4(1)$:\;\;\;\;\; None, 
\item[(ii)] $S^{4}(1)$: $S^{3}(\frac{1}{\sqrt{2}})$, or $S^1(\frac{1}{\sqrt{2}})\times S^{2}(\frac{1}{\sqrt{2}})$ . 
\end{itemize}
\item[II.] {\em Pseudo-Riemannian Cases of index 1}:
\begin{itemize}
\item[(i)] $\mathbb{R}^4_1$:\;\;\;\;\; None, 
\item[(ii)] $S_1^{4}(1)$: $S_1^{3}(\frac{1}{\sqrt{2}})$, $S^1(\frac{1}{\sqrt{2}})\times S_1^{2}(\frac{1}{\sqrt{2}})$, or $S_1^1(\frac{1}{\sqrt{2}})\times S^{2}(\frac{1}{\sqrt{2}})$ . 
\item[(iii)] $H_1^{4}(1)$: $H^{3}(\frac{1}{\sqrt{2}})$, $H^{1}(\frac{1}{\sqrt{2}})\times H^{2}(\frac{1}{\sqrt{2}})$, or $H^{1}(\frac{1}{\sqrt{2}})\times H^{2}(\frac{1}{\sqrt{2}})$ . 
\end{itemize}
\item[III.] {\em Pseudo-Riemannian Cases of index 2}:
\begin{itemize}
\item[(i)] $\mathbb{R}^4_2$:\;\;\;\;\; None, 
\item[(ii)] $S_2^{4}(1)$: $S_2^{3}(\frac{1}{\sqrt{2}})$, $S_1^1(\frac{1}{\sqrt{2}})\times S_1^{2}(\frac{1}{\sqrt{2}})$, or $S_1^1(\frac{1}{\sqrt{2}})\times S_1^{2}(\frac{1}{\sqrt{2}})$ . 
\item[(iii)] $H_2^{4}(1)$: $H_1^{3}(\frac{1}{\sqrt{2}})$, $H^{1}(\frac{1}{\sqrt{2}})\times H_1^{2}(\frac{1}{\sqrt{2}})$, or $H_1^{1}(\frac{1}{\sqrt{2}})\times H^{2}(\frac{1}{\sqrt{2}})$ . 
\end{itemize}
\item[IV.] {\em Pseudo-Riemannian Cases of index 3}:
\begin{itemize}
\item[(i)] $\mathbb{R}^4_3$:\;\;\;\;\; None, 
\item[(ii)] $S_3^{4}(1)$: $S_3^{3}(\frac{1}{\sqrt{2}})$, $S_1^1(\frac{1}{\sqrt{2}})\times S_2^{2}(\frac{1}{\sqrt{2}})$. 
\item[(iii)] $H_3^{4}(1)$: $H_2^{3}(\frac{1}{\sqrt{2}})$, $H^{1}(\frac{1}{\sqrt{2}})\times H_2^{2}(\frac{1}{\sqrt{2}})$, or $H_1^{1}(\frac{1}{\sqrt{2}})\times H_1^{2}(\frac{1}{\sqrt{2}})$ . 
\end{itemize}
\end{itemize}
\end{example}
The complete classifications given in Corollary \ref{SSn} reveal the similarity between the families of proper biharmonic \sdo hypersurfaces of \sdo space form and the classifications of proper biharmonic hypersurfaces in $S^{n+1}(1)$ obtained by Caddeo-Montaldo-Oniciuc in \cite{BMO1}. An important application of this lies in the fact that the two special proper biharmonic \sdo hypersurfaces classified in Corollary \ref{SSn}, as in the Riemannian cases, can help us to further study the constructions and classifications of proper biharmonic \sdo submanifolds in \sdo space forms.
\begin{theorem}\label{GZ}
A minimal submanifold $M^m_r\longrightarrow H^n_{s-1}(a)\times \{b\}$ with $a^2+b^2=1, a\in (0, 1)$ is a proper biharmonic submanifold of $H^{n+1}_{s}(1)$ if and only if $a=1/\sqrt{2}$ and $b=\pm 1/\sqrt{2}$. In particular, any minimal submanifold $M^m\longrightarrow H^n(1/\sqrt{2})$ of hyperbolic space gives rise to a proper biharmonic submanifold of Anti-de Sitter space $H^{n+1}_{1}(1)$ via composition: $M^m\longrightarrow H^n(1/\sqrt{2})\longrightarrow H^{n+1}_{1}(1)$.
\end{theorem}
\begin{proof}
Let $\{x^i\}$ be the rectangular coordinates and $\langle, \rangle$ denote the pseudo-Euclidean product of $\r^{n+2}_{s+1}=\r^{n+1}_{s}\times \r^1_1$. It is easy to see that the tangent vector fields of $H^n_{s-1}(a)\times \{b\}$ can be described as $$\{X=(X^1,\cdots, X^{n+1},0)\in \r^{n+2}_{s+1}| \sum_{i=1}^{n+1}x^iX^i=0 \}.$$
Let $p=(x^1, \cdots, x^{n+1},b)$ and $\xi=(x^1, \cdots, x^{n+1}, -a^2/b)$, then one can check that
\begin{equation}\notag
\langle \xi, X\rangle=0,\;\;\langle \xi, p\rangle=0,\;\;\langle \xi, \xi\rangle=-a^2-a^4/b^2<0,
\end{equation}
and hence $\eta=\frac{1}{c}\xi$ (with $c^2=a^2+a^4/b^2$) is a time-like unit normal vector field of $H^n_{s-1}(a)\times \{b\}$ in $H^{n+1}_s(1)$. Using Weingarten and Gauss equations we compute
\begin{eqnarray}\notag
\nabla^{H^{n+1}_s}_X\eta &=& \nabla^{\perp}_X\eta-A_{\eta}(X)\\\notag
&=& \frac{1}{c}\nabla^{H^{n+1}_s}_X\xi=\frac{1}{c}\{\nabla^{\r^{n+2}_{s+1}}_X\xi-\langle \xi, X\rangle\,p\}\\\label{31}
&=& \frac{1}{c}\nabla^{\r^{n+2}_s}_{(X^1,\cdots, X^{n+1},0)}\,(x^1, \cdots, x^{n+1},-a^2/b)\\\notag
&=& \frac{1}{c} X,
\end{eqnarray}
where the third equality was obtained by using the fact that the second fundamental form of $H^{n+1}_s(1)$ in $\r^{n+2}_{s+1}$ is given by $b(X, Y)= -\langle {\tilde A}_p(X), Y\rangle\,p=-\langle -X, Y\rangle\,p=\langle X, Y\rangle\,p$ (cf. \cite{Ch0}).\\
It follows from (\ref{31}) that $\nabla^{\perp}_X\eta=0$ and $A_{\eta}(X)=-\frac{1}{c} X$.\\
\indent We introduce the notations $\phi: M^m_r\longrightarrow H^n_{s-1}(a)\times \{b\}$ and $j:H^n_{s-1}(a)\times \{b\}\longrightarrow H^{n+1}_s(1)$ for the isometric immersions of the submanifolds. Then, a straightforward computation gives the mean curvature vector $H$ of $M^m_r$ in $H^{n+1}_s(1)$ as
\begin{eqnarray}
H(\phi\circ j) &=& H(\phi)+\frac{1}{m}\sum_{i=1}^m\epsilon_i b(X_i, X_i)\\\notag
&=& H(\phi)-\langle \frac{1}{m}\sum_{i=1}^m\epsilon_i b(X_i, X_i), \eta\rangle \,\eta\\\notag
&=& H(\phi)- \frac{1}{m}\sum_{i=1}^m\epsilon_i\langle A_{\eta}(X_i), X_i\rangle \,\eta\\\notag
&=& H(\phi)+ \frac{1}{cm}\sum_{i=1}^m\epsilon_i \langle X_i, X_i\rangle \,\eta=H(\phi)+ \frac{1}{c}\eta.
\end{eqnarray}
\indent By the assumption that $\phi: M^m_r\longrightarrow H^n_{s-1}(a)\times \{b\}$ is minimal we conclude that the mean curvature vector field of $M^m_r$ in $H^{n+1}_s(1)$ is given by 
\begin{equation}
H(\phi\circ j) = \frac{1}{c}\eta.
\end{equation}
\indent It follows that $A_H(X)=\frac{1}{c}A_{\eta}(X)=-\frac{1}{c^2}X$ for any $X$ tangent to $M^m_r$ and hence, $M^m_r$ is a pseudo-umbilical non-minimal submanifold $H^{n+1}_s(1)$ with constant mean curvature vector field $H=\frac{1}{c}\eta$. Therefore, as in the proof of Theorem \ref{MT2}, the second equation of (\ref{BSMC}) is satisfied. Substituting $\nabla^{\perp}_XH=0$ and $A_H(X)=-\frac{1}{c^2}X$ into the first equation of (\ref{BSMC}) we have
\begin{equation}
0=\sum_{i=1}^m\,\epsilon_i B(A_H(e_i), e_i)+mH=-\frac{1}{c^2}\sum_{i=1}^m\epsilon_iB(e_i, e_i)+mH=(1-\frac{1}{c^2})mH.
\end{equation}
Solving the equation we have $c^2=1$, i.e., $a^2+a^4/b^2=1$ which, together with $a^2+b^2=1$, gives $a=1/\sqrt{2}, b=\pm 1/\sqrt{2}$. Thus, we obtain the theorem.
\end{proof}
A similar proof gives the following construction of proper biharmonic submanifolds of \sdo spheres.
\begin{theorem}\label{GZ1}
A minimal submanifold $M^m_r\longrightarrow S^n_s(a)\times \{b\}$ with $a^2+b^2=1, a\in (0, 1)$ is a proper biharmonic submanifold of $S^{n+1}_s(1)$ if and only if $a=1/\sqrt{2}$ and $b=\pm 1/\sqrt{2}$. In particular, any minimal submanifold $M^m_r\longrightarrow S^n_1(1/\sqrt{2})\; (r=0, 1)$ of \sdo sphere gives rise to a proper biharmonic submanifold of the de Sitter space $S^{n+1}_{1}(1)$ via composition: $M_r^m\longrightarrow S_1^n(1/\sqrt{2})\longrightarrow S^{n+1}_{1}(1)$.
\end{theorem}
Also, a proof similar to that of Theorem \ref{GZ} and to the proof of Theorem 3.11 in \cite{CMO2} in the Riemannian case gives
\begin{theorem}\label{GZ2}
(i) Let $\phi:M^m_r\longrightarrow H^p_{t-1}(\frac{1}{\sqrt{2}})$ and $\varphi: N^q_l\longrightarrow H^{n-p}_{s-t}(\frac{1}{\sqrt{2}})$ be minimal submanifolds. Then, $\Phi: M^m_r\times N^q_l\longrightarrow H^{n+1}_s(1)$ with $\Phi(x,y)=(\phi(x), \varphi(y))$ is a proper biharmonic submanifold of $ H^{n+1}_s(1)$.\\
(ii) Let $\phi:M^m_r\longrightarrow S^p_t(\frac{1}{\sqrt{2}})$ and $\varphi: N^q_l\longrightarrow S^{n-p}_{s-t}(\frac{1}{\sqrt{2}})$ be minimal submanifolds. Then, $\Phi: M^m_r\times N^q_l\longrightarrow S^{n+1}_s(1)$ with $\Phi(x,y)=(\phi(x), \varphi(y))$ is a proper biharmonic submanifold of $ S^{n+1}_s(1)$.
\end{theorem}
\section{Biharmonic hypersurfaces of Riemannian space forms vs. space-like biharmonic hypersurfaces of Lorentzian space forms}
An $n$-dimensional Lorentzian manifold is a pseudo-Riemannian manifold of index $(1, n-1)$ (or equivalently $(n-1, 1)$). This special type of pseudo-Riemannian manifolds are of great importance to general relativity as spacetimes are modeled as $4$-dimensional Lorentzian manifolds. In particular, Lorentzian space forms $\r^n_1$, $S^n_1(C)$, $H^n_1(C) $ are well known to be called Minkowski, de Sitter, and Anti-de Sitter space-times respectively. On the other hand, space-like submanifolds usually appear in the study of problems related to causality in the theory of general relativity. Also, space-like hypersurfaces with constant mean curvature are convenient as initial hypersurfaces for the Cauchy problem in arbitrary spacetime and for studying the prolongation of gravitational radiation (See e.g., \cite{BCM} for details). Based on these, we take a closer look at the space-like biharmonic hypersurfaces in Lorentzian space forms in this section. Some comparisons of biharmonic hypersurfaces of Riemannian space forms and space-like biharmonic hypersurfaces of Lorentzian space forms are made.\\
\indent Recall that a submanifold of a pseudo-Riemannian manifold is {\em space-like} if the induced metric on the submanifold is a Riemannian metric. It is well known that a hypersurface $M^n\longrightarrow N^{n+1}_1$ of a Lorentzian manifold is space-like if and only if its unit normal vector field $\xi$ satisfies $\langle \xi, \xi\rangle=-1$.\\
\indent Let $\varphi:M^{m}\longrightarrow N^{m+1}_1$ be an isometric immersion
of a space-like hypersurface into a Lorentzian manifold. We denote by $A$ the shape operator of
$\varphi$ with respect to $\xi$, a unit normal vector field to $\varphi(M)
\subset N$, and by $H=f\xi$ the mean curvature vector field of
$\varphi$ (f the mean curvature function). Then we have
\begin{theorem}\label{PE114}
Let $\varphi:M^{m}\longrightarrow N^{m+1}_1$ be a space-like hypersurface in a Lorentzian manifold with mean curvature vector field $H=f\xi$. Then
$\varphi$ is biharmonic if and only if
\begin{equation}
\begin{cases}
\Delta f+f |A|^{2}-f{\rm
Ric}^N(\xi,\xi)=0,\\
A\,({\rm grad}\,f) -\frac{m}{2} f{\rm grad}\, f
-\, f \,({\rm Ric}^N\,(\xi))^{\top}=0,
\end{cases}
\end{equation}
where $Ric : T_pM\longrightarrow T_pM$ denotes the Ricci operator of
the ambient space defined by $\langle Ric\, (Z), W\rangle={\rm
Ric}^N (Z, W)$.
\end{theorem}
\begin{proof}
This follows from Corollary \ref{CD1} with $\epsilon_{m+1}=-1$ since $\langle \xi, \xi\rangle=-1$ for a space-like hypersurface.
\end{proof}
As a straightforward consequence, we have
\begin{corollary}
A space-like hypersurface in a Lorentzian Einstein space $(N^{m+1}_1,h)$ is biharmonic if
and only if its mean curvature function $f$ is a solution of the
following PDEs
\begin{equation}\label{Ein}
\begin{cases}
\Delta f+f\,|A|^{2}+ \frac{r}{m+1}f=0,\\
A(\,{\rm grad}\,f) -\frac{m}{2}\,f{\rm grad}\, f
=0,
\end{cases}
\end{equation}
In particular, a space-like hypersurface $\varphi :(M^{m}, g)\longrightarrow
(N^{m+1}_1(C), h)$ in a Lorentzian space of constant sectional curvature $C$ is
biharmonic if and only if its mean curvature function $f$ is a
solution of the following PDEs
\begin{equation}\label{M1}
\begin{cases}
\Delta f+f\,|A|^{2}+ mCf=0,\\
A\,({\rm grad}\,f) -\frac{m}{2}\, f{\rm grad}\, f=0.
\end{cases}
\end{equation}
\end{corollary}
\begin{proof}
It is well known that if $(N^{m+1},h)$ is an Einstein manifold then
${\rm Ric} (Z,W)=\frac{r}{m+1}h(Z,W)$ for any $Z, W \in TN$ and
hence $(Ric\,(\xi))^{\top}=0$ and ${\rm
Ric}^N(\xi,\xi)=\frac{r}{m+1}$. When $(N^{m+1}(C),h)$ is a space of
constant sectional curvature $C$, then it is an Einstein space with
the scalar curvature $r=m(m+1)C$. From these the corollary follows.
\end{proof}
\begin{remark}
We remark that a hypersurface $\varphi :(M^{m}, g)\longrightarrow
(N^{m+1}(C), h)$ in a Riemannian space form of constant sectional curvature $C$ is
biharmonic if and only if its mean curvature function $f$ is a
solution of the following PDEs, which was obtained by different
authors in several steps (see \cite{Ji2}, \cite{Ch1} and \cite{CMO2})
\begin{equation}
\begin{cases}
\Delta f-f\,|A|^{2}+ mCf=0,\\
A\,({\rm grad}\,f) +\frac{m}{2}\, f{\rm grad}\, f=0.
\end{cases}
\end{equation}
So, the biharmonic equation for space-like hypersurfaces in Lorentzian space forms differs from the biharmonic equation for hypersurfaces in Riemannian space forms by a minus sign in the second term of each equation in the system.
\end{remark}
Using Corollary \ref{SSn} we have the following classification of space-like biharmonic hypersurfaces with at most two distinct principal curvatures.
\begin{corollary}
$($\em I$)$ A space-like biharmonic hypersurface in a Lorentzian space form $S^{n+1}_1(1)$ or $\r^{n+1}_1$ with at most two distinct principal curvatures is maximal, i.e., $H=0$; \\
$($ II $)$ A space-like biharmonic hypersurface in a Lorentzian space form $H^{n+1}_1(1)$ with at most two distinct principal curvatures is either a part of either $H^{n}(\frac{1}{\sqrt{2}})$ or a part of $H^{p}(\frac{1}{\sqrt{2}})\times H^{q}(\frac{1}{\sqrt{2}})$ with $p+q=n, p\ne q$.
\end{corollary}
\begin{remark}
Based on what we know so far, the set of proper biharmonic hypersurfaces in the Riemannian space form $S^m(1)$ (respectively, in $\r^m$ and $H^m(1)$) and the set of space-like proper biharmonic hypersurfaces in the Lorentzian space form $H^m_1(1)$ (respectively, in $\r^m_1$ and $S^m_1(1)$) have the same ``picture". It would be interesting to know if the systems
\begin{eqnarray}
\begin{cases}
\Delta f-f\,|A|^{2}+ mf=0,\\
A\,({\rm grad}\,f) +\frac{m}{2}\, f{\rm grad}\, f=0, \;\;
\end{cases}
{\rm and}\;\; \begin{cases}
\Delta f+f\,|A|^{2}- mf=0,\\
A\,({\rm grad}\,f) -\frac{m}{2}\, f{\rm grad}\, f=0
\end{cases}
\end{eqnarray}
have the same solution set. Similarly, one would like to know if the systems
\begin{eqnarray}
\begin{cases}
\Delta f-f\,|A|^{2}=0,\\
A\,({\rm grad}\,f) +\frac{m}{2}\, f{\rm grad}\, f=0, \;\;
\end{cases}
{\rm and}\;\; \begin{cases}
\Delta f+f\,|A|^{2}=0,\\
A\,({\rm grad}\,f) -\frac{m}{2}\, f{\rm grad}\, f=0
\end{cases}
\end{eqnarray}
have the same solution set. 
\end{remark}


\begin{thebibliography}{99}
\bibitem{BMO1} A. Balmus, S. Montaldo and C. Oniciuc, {\em Classification results for biharmonic submanifolds in spheres}, Israel J. Math. 168 (2008), 201--220.
\bibitem{BMO2} A. Balmus, S. Montaldo, C. Oniciuc, {\em Biharmonic PNMC Submanifolds in Spheres}, Ark. Mat. 51(2013), 197-221.
\bibitem{BCM}  A. Brasil, Rosa M. B.  Chaves, and M.  Mariano, {\em Complete spacelike submanifolds with parallel mean curvature vector in a semi-Riemannian space form},  J. Geom. Phys. 56 (2006), no. 10, 2177-2188.
\bibitem{CMO2} R. Caddeo, S. Montaldo and C. Oniciuc, {\em Biharmonic submanifolds in spheres},
 Israel J. Math. 130 (2002), 109--123.
\bibitem{Ch0} B. -Y. Chen, {\em Finite type submanifolds in pseudo-Euclidean spaces and applications}, Kodai Math. J., 8 (1985), 358-374.
\bibitem{Ch1} B. -Y. Chen, {\em Some open problems and conjectures on submanifolds of finite
type}, Soochow J. Math. 17 (1991), no. 2, 169--188.
\bibitem{Ch2} B. -Y. Chen, {\em Classification of marginally trapped Lorentzian flat surfaces in $E^4_2$ and its application to biharmonic surfaces},  J. Math. Anal. Appl. 340 (2008), no. 2, 861-875.
\bibitem{Ch3} B. Y. Chen, {\em Total Mean Curvature and Submanifolds of Finite Type}, 2nd Edition, World Scientific, 2015.
\bibitem{CI} B. Y. Chen and S. Ishikawa, {\em Biharmonic pseudo-Riemannian submanifolds
in pseudo-Euclidean spaces}, Kyushu J. Math. 52 (1998), no. 1,
167--185.
\bibitem{DKP} F. Defever, G. Kaimakamis and V. Papantoniou, {\em Biharmonic hypersurfaces of the $4$-dimensional
semi-Euclidean space $\mathbb{E}^4_s$}, J. Math. Anal. Appl. 315 (2006) 276-286.
\bibitem{Di} I. Dimitri\'c, {\em Submanifolds of $E\sp m$ with harmonic mean curvature vector}, Bull. Inst. Math. Acad. Sinica 20 (1992), no. 1, 53--65.
\bibitem{EL} J. Eells and L. Lemaire, {\em Selected topics in harmonic maps}, CBMS, 50, Amer. Math. Soc. (1983).
\bibitem{Fu} Y. Fu, {\em Biharmonic submanifolds with parallel mean curvature vector in pseudo-Euclidean spaces}, Math. Phys. Anal. Geom. 16 (2013), no. 4, 331-344.
\bibitem{Ji1} G. Y. Jiang, {\em $2$-harmonic isometric immersions between Riemannian manifolds}. Chinese Ann. Math. Ser. A, 7 (2) (1986), 130--144.
\bibitem{Ji2} G. Y. Jiang, {\em $2$-Harmonic maps and their first and second variational formulas}, Chin. Ann. Math. Ser. A, 7(4) (1986) 389-402.
\bibitem{Ji3} G. Y. Jiang, {\em Some non-existence theorems of $2$-harmonic isometric immersions into Euclidean spaces}, Chin. Ann. Math. Ser. A, 8 (1987) 376-383.
\bibitem{LD} J. Liu and L. Du {\em Classification of proper biharmonic hypersurfaces in pseudo-Riemannian space forms},  Differential Geom. Appl. 41 (2015), 110-122. 
\bibitem{LDZ} J. Liu, L. Du and J. Zhang, {\em Minimality on biharmonic space-like submanifolds in pseudo-Riemannian space forms}  J. Geom. Phys. 92 (2015), 69-77.
\bibitem{MOR} S. Montaldo, C. Oniciuc, A. Ratto, {\em Proper biconservative immersions into the Euclidean space}, Ann. di Mate. Pura ed Applicata, DOI:10.1007/s10231-014-0469-4.
\bibitem{Ou1} Y. -L. Ou, {\em Biharmonic hypersurfaces in Riemannian manifolds}, Pacific J. of Math, 248 (1), (2010), 217-232.
\bibitem{Ou2} Y. -L. Ou, {\em Some recent progress of biharmonic submanifolds},   arXiv:1511.09103, to appear in AMS Contemporary Mathematics, 2016.
\bibitem{Oy} C. Z. Ouyang, {\em  2-harmonic space-like submanifolds of a pseudo-Riemannian space form}, (Chinese),  Chinese Ann. Math. Ser. A 21 (2000), no. 6, 649-654. 
\bibitem{Sa} T. Sasahara, {\em Biharmonic submanifolds in nonflat Lorentz 3-space forms}, Bull. Aust. Math. Soc. 85 (2012), no. 3, 422-432. 
\bibitem{Zh}  W. Zhang, {\em Biharmonic space-like hypersurfaces in pseudo-Riemannian space}, preprint, arXiv:0808.1346, 2008.
\end{thebibliography}
\end{document}